\documentclass[12pt]{amsproc}

\newtheorem{theorem}{\bf Theorem}[section]
\newtheorem{lemma}[theorem]{\bf Lemma}
\newtheorem{proposition}[theorem]{\bf Proposition}
\newtheorem{corollary}[theorem]{\bf Corollary}

\author[C. Acciarri]{Cristina Acciarri}
\address{Department of Mathematics, University of Brasilia, 70910-900 Bras\'ilia DF, Brazil}
\email{acciarricristina@yahoo.it}

\author[P. Shumyatsky]{Pavel Shumyatsky}
\address{Department of Mathematics, University of Brasilia, 70910-900 Bras\'ilia DF, Brazil}
\email{pavel@unb.br}

\author[D. San\c c\~ao da Silveira ]{Danilo San\c c\~ao da Silveira}
\address{Department of Mathematics, Federal University of Goi\'as, 75704-020 Catal\~ao GO, Brazil}
\email{sancaodanilo@ufg.br}


\thanks{This work was supported by the Conselho Nacional de Desenvolvimento Cient\'{\i}fico e Tecnol\'ogico (CNPq),  and Funda\c c\~ao de Apoio \`a Pesquisa do Distrito Federal (FAPDF), Brazil.}

\title[On groups with automorphisms]{On groups with automorphisms\\ whose fixed points are Engel}

\begin{document}

\begin{abstract} We complete the study of finite and profinite groups admitting an action by an elementary abelian group under which the centralizers of automorphisms consist of Engel elements. In particular, we prove the following theorems. 

Let $q$ be a prime and $A$ an elementary abelian $q$-group of order at least $q^2$ acting coprimely on a profinite group $G$. Assume that all elements in $C_{G}(a)$ are Engel in $G$ for each $a\in A^{\#}$. Then $G$ is locally nilpotent (Theorem B2).

Let $q$ be a prime, $n$ a positive integer and $A$ an elementary abelian group of order $q^3$ acting coprimely on a finite group $G$. Assume that for each $a\in A^{\#}$ every element of $C_{G}(a)$ is $n$-Engel in $C_{G}(a)$. Then the group $G$ is $k$-Engel for some $\{n,q\}$-bounded number $k$ (Theorem A3).
\end{abstract}

\maketitle
\section{Introduction}
Let $A$ be a  finite group acting on a finite group $G$. Many well-known results show that the structure of the centralizer $C_G(A)$ (the fixed-point subgroup) of $A$ has influence over the structure of $G$. The influence is especially strong if $(|A|,|G|)=1$, that is, the action of $A$ on $G$ is coprime. Following the solution of the restricted Burnside problem it was discovered that the exponent of $C_G(A)$ may have strong impact over the exponent of $G$. Recall that a group $G$ is said to have exponent $n$ if $x^n=1$ for every $x\in G$ and $n$ is the minimal positive integer with this property. The next result was obtained in \cite{KS}.
\medskip

{\bf Theorem\  A1.} {\it Let $q$ be a prime, $n$ a positive integer and $A$ an elementary abelian group of order $q^2$. Suppose that $A$ acts coprimely on a finite group $G$ and assume that $C_{G}(a)$ has exponent dividing $n$ for each $a\in A^{\#}$. Then the exponent of $G$ is $\{n,q\}$-bounded.}
\medskip
  
Here and throughout the paper $A^{\#}$ denotes the set of nontrivial elements of $A$. We use the expression ``$\{a,b,\dots\}$-bounded'' to abbreviate ``bounded from above in terms of  $a,b,\dots$ only''. The proof of  the above result involves a number of deep ideas. In particular, Lie-theoretical results of Zelmanov obtained in his solution of the restricted Burnside problem \cite{Z0,Z1} are combined with the Lubotzky--Mann theory of powerful $p$-groups \cite{luma}, Lazard's criterion for a pro-$p$ group to be $p$-adic analytic \cite{L}, and a theorem of Bahturin and Zaicev on Lie algebras admitting a group of automorphisms whose fixed-point subalgebra is PI \cite{BZ}.

A profinite (non-quantitative) version of the above theorem was obtained in \cite{eu}. By an automorphism of a profinite group we always mean a continuous automorphism. A group $A$ of automorphisms of a profinite group $G$ is coprime if $A$ is finite while $G$ is an inverse limit of finite groups whose orders are relatively prime to the order of $A$. A group is said to be locally finite if each of its finitely generated subgroups is finite. A group $G$ is torsion if each element in $G$ has finite order.  Using a reduction theorem of Wilson \cite{W2}, Zelmanov showed that a profinite group is locally finite if and only if it is torsion \cite{zelmanov}. The profinite (non-quantitative) version of Theorem A1 is as follows.
\medskip

{\bf Theorem\  B1.} {\it Let $q$ be a prime and $A$ an elementary abelian $q$-group of order at least $q^2$. Suppose that $A$ acts coprimely on a profinite group $G$ and assume that $C_{G}(a)$ is torsion for each $a\in A^{\#}$. Then $G$ is locally finite.}
\medskip

In the recent work \cite{shusa} the situation where the centralizers $C_{G}(a)$ consist of Engel elements was dealt with. If  $x,y$ are elements of a (possibly infinite) group $G$, the commutators $[x,_n y]$ are defined inductively by the rule
$$[x,_0 y]=x,\quad [x,_n y]=[[x,_{n-1} y],y]\quad \text{for all}\, n\geq 1.$$
An element $x$ is called a (left) Engel element if for any $g\in G$ there exists $n$, depending on $x$ and $g$, such that $[g,_n x]=1$.  A group $G$ is called Engel if all elements of $G$ are Engel. The element $x$ is called a (left) $n$-Engel element if for any $g\in G$ we have $[g,_n x]=1$. The group $G$ is $n$-Engel if all elements of $G$ are $n$-Engel. The following result was proved in \cite{shusa}.
\medskip

{\bf Theorem\  A2.} {\it Let $q$ be a prime, $n$ a positive integer and $A$ an elementary abelian group of order $q^2$. Suppose that $A$ acts coprimely on a finite group $G$ and assume that for each $a\in A^{\#}$ every element of $C_{G}(a)$ is $n$-Engel in $G$. Then the group $G$ is $k$-Engel for some $\{n,q\}$-bounded number $k$.}
\medskip

One of the goals of the present article is to establish a profinite (non-quantitative) version of Theorem A2. A group is said to be locally nilpotent if each of its finitely generated subgroups is nilpotent. An important theorem of Wilson and Zelmanov \cite[Theorem 5]{WZ}  tells us that a profinite group is locally nilpotent if and only if it is Engel.
\medskip

{\bf Theorem\  B2.} {\it Let $q$ be a prime and $A$ an elementary abelian $q$-group of order at least $q^2$. Suppose that $A$ acts coprimely on a profinite group $G$ and assume that all elements in $C_{G}(a)$ are Engel in $G$ for each $a\in A^{\#}$. Then $G$ is locally nilpotent.}
\medskip

If, in Theorem A2, we relax the hypothesis that every element of $C_{G}(a)$ is $n$-Engel in $G$ and require instead that every element of $C_{G}(a)$ is $n$-Engel in $C_{G}(a)$, we quickly see that the result is no longer true. An example of a finite non-nilpotent group $G$ admitting a four-group of automorphisms $A$ such that $C_G(a)$ is abelian for each $a\in A^{\#}$ can be found for instance in \cite{PS-CA3}. On the other hand, in the present article we establish the following theorem.
\medskip

{\bf Theorem\  A3.} {\it Let $q$ be a prime, $n$ a positive integer and $A$ an elementary abelian group of order $q^3$. Suppose that $A$ acts coprimely on a finite group $G$ and assume that for each $a\in A^{\#}$ every element of $C_{G}(a)$ is $n$-Engel in $C_{G}(a)$. Then the group $G$ is $k$-Engel for some $\{n,q\}$-bounded number $k$.}
\medskip

A profinite (non-quantitative) version of the above theorem already exists in the literature. It was obtained in \cite{PS-CA3}:
\medskip

{\bf Theorem\  B3.} {\it Let $q$ be a prime and $A$ an elementary abelian $q$-group of order at least $q^3$. Suppose that $A$ acts coprimely on a profinite group $G$ and assume that $C_{G}(a)$ is locally nilpotent for each $a\in A^{\#}$. Then $G$ is locally nilpotent.}
\medskip

Thus, the purpose of the present article is to provide the proofs for Theorems B2 and A3. As the reader will see, the proofs have rather a lot in common. In the same time, by obvious reasons, the differences in the proofs are rather significant. 

Throughout the paper we use, without special references, the well-known properties of coprime actions (see for example \cite[Lemma 3.2]{PSprofinite}).

If $\alpha$ is a coprime automorphism of a profinite group $G$, then $C_{G/N}(\alpha)=C_G(\alpha)N/N$ for any $\alpha$-invariant normal subgroup $N$.
 
If $A$ is a noncyclic abelian group acting coprimely on a profinite group $G$, then $G$ is generated by the subgroups $C_G(B)$, where $A/B$ is cyclic.

        \section{About Lie algebras and Lie rings}
Let $X$ be a subset of a Lie algebra $L$. By a commutator in elements of $X$ we mean any element of $L$ that can be obtained as a Lie product of elements of $X$ with some system of brackets. If $x_1,\ldots,x_k,x, y$ are elements of $L$, we define inductively 
$$[x_1]=x_1; [x_1,\ldots,x_k]=[[x_1,\ldots,x_{k-1}],x_k]$$
and 
$[x,_0y]=x; [x,_my]=[[x,_{m-1}y],y],$ for all positive integers $k,m$.  
As usual, we say that an element $a\in L$ is ad-nilpotent if there exists a positive integer $n$ such that $[x,_na]=0$ for all $x\in L$. If $n$ is the least integer with the above property, then we say that $a$ is ad-nilpotent of index $n$. 

The next theorem represents the most general form of the Lie-theoretical part of the solution of the restricted Burnside problem. It was announced by Zelmanov in \cite{Z1}. A detailed proof was recently published in \cite{zenew}.

\begin{theorem}\label{Z1992}
Let $L$ be a Lie algebra over a field and suppose that $L$ satisfies a polynomial identity. If $L$ can be generated by a finite set $X$ such that every commutator in elements of $X$ is ad-nilpotent, then $L$ is nilpotent.
\end{theorem}

The next theorem, which was proved by Bahturin and Zaicev for soluble groups $A$ \cite{BZ} and later extended by Linchenko to the general case \cite{l}, provides an important criterion for a Lie algebra to satisfy a polynomial identity. 
 
\begin{theorem}\label{blz}
Let $L$ be a Lie algebra over a field $K$. Assume that a finite group $A$ acts on $L$ by automorphisms in such a manner that $C_L(A)$ satisfies a polynomial identity. Assume further that the characteristic of $K$ is either $0$ or prime to the order of $A$. Then $L$ satisfies a polynomial identity.
\end{theorem}

Both Theorems \ref{Z1992} and  \ref{blz} admit  the following respective quantitative versions (see for example \cite{aaaa}).

\begin{theorem}\label{Z1}
Let $L$ be a Lie algebra over a field $K$ generated by $a_1,\ldots,a_m$. Suppose that $L$ satisfies a polynomial identity $f\equiv 0$ and each commutator in $a_1,\ldots,a_m$ is ad-nilpotent of index at most $n$. Then $L$ is nilpotent of $\{f,K,m,n\}$-bounded class.
\end{theorem}

\begin{theorem}\label{bZ1}
 Let $L$ be as in Theorem \ref{blz} and assume that $C_L(A)$ satisfies a polynomial identity $f\equiv 0$. Then $L$ satisfies a polynomial identity of $\{|A|,f,K\}$-bounded degree.
\end{theorem}
The combination of Theorems \ref{Z1} and \ref{bZ1} yields the following corollary.
\begin{corollary}\label{ZelBazaField}
Let $L$ be a Lie algebra over a field $K$ and $A$ a finite group of automorphisms of $L$ such that $C_L(A)$ satisfies the polynomial identity $f\equiv 0$. Suppose that the characteristic of $K$ is either $0$ or prime to the order of $A$. Assume that $L$ is generated by an $A$-invariant set of $m$ elements in which every commutator is ad-nilpotent of index at most $n$. Then $L$ is nilpotent of $\{|A|,f,K,m,n\}$-bounded class.
\end{corollary}

In the present article we will be required to work with Lie rings, and not only with Lie algebras. As usual, $\gamma_i(L)$ denotes the $i$th term of the lower central series of $L$. In \cite{shusa} the following result for Lie rings was established.

\begin{theorem}\label{BazaRing} Let $L$ be a Lie ring and $A$ a finite group of automorphisms of $L$ such that $C_L(A)$ satisfies the polynomial identity $f\equiv 0$. Further, assume that $L$ is generated by an $A$-invariant set of $m$ elements such that every commutator in the generators is ad-nilpotent of index at most $n$. Then there exists positive integers $e$ and $c$, depending only on $|A|, f, m$ and $n$, such that  $e\gamma_c(L)=0$.
\end{theorem}

Another useful lemma, whose proof can be found in \cite{KS}, is the following. 

\begin{lemma}\label{lemmanovo}
Let $L$ be a Lie ring and $H$ a subring of $L$ generated by $m$ elements $h_{1},\ldots, h_{m}$ such that all commutators in $ h_i$ are  ad-nilpotent in $L$ of index at most $n$. If $H$ is nilpotent of class $c$, then for some $\{c, m,n\}$-bounded number $u$ we have  $[L,\underbrace{H,\ldots, H}_u]=0.$
\end{lemma}

Recall that the identity $$\sum_{\sigma\in S_n}[y,x_{\sigma(1)},\ldots, x_{\sigma(n)}]\equiv 0$$ is called the linearized $n$-Engel identity. In general, Theorem \ref{Z1992} cannot be extended to the case where $L$ is just a Lie ring (rather than a Lie algebra over a field). However such an extension does hold in the particular case where the polynomial identity $f\equiv 0$ is a linearized Engel identity.  More precisely, by combining Theorems \ref{BazaRing}, \ref{Z1} and Lemma \ref{lemmanovo}  the following result can be obtained. See \cite{shusa} for further details.

\begin{theorem}\label{Zelring} Let $F$ be the free Lie ring and $f$ an element of $F$ (Lie polynomial) such that $f\not\in pF$ for any prime $p$. Suppose that $L$ is a Lie ring generated by finitely many elements $ a_1,\ldots, a_m$ such that all commutators in the generators are ad-nilpotent of index at most $n$. Assume that $L$ satisfies the identity $f\equiv 0$. Then $L$ is nilpotent with $\{f,m,n\}$-bounded class.
\end{theorem}

      \section{On associated Lie rings}
There are several well-known ways to associate a Lie ring to a group (see \cite{Huppert2,Khu1,aaaa}). For the reader's convenience we will briefly describe the construction that we are using in the present paper.

Let $G$ be a group. A series of subgroups $$G=G_1\geq G_2\geq\dots\eqno{(*)}$$ is called an $N$-series if it satisfies $[G_i,G_j]\leq G_{i+j}$ for all $i,j$. Obviously any $N$-series is central, i.e. $G_i/G_{i+1}\leq Z(G/G_{i+1})$ for any $i$.  Given an $N$-series $(*)$, let $L^*(G)$ be the direct sum of the abelian groups $L_i^*=G_i/G_{i+1}$, written additively. Commutation in $G$ induces a binary operation $[,]$ in $L^*(G)$. For homogeneous elements $xG_{i+1}\in L_i^*,yG_{j+1}\in L_j^*$ the operation is defined by $$[xG_{i+1},yG_{j+1}]=[x,y]G_{i+j+1}\in L_{i+j}^*$$ and extended to arbitrary elements of $L^*(G)$ by linearity. It is easy to check that the operation is well-defined and that $L^*(G)$ with the operations $+$ and $[,]$ is a Lie ring. 

An $N$-series is called an $N_{p}$-series if  $G_{i}^{p}\leq G_{pi}$ for all $i$.
If all quotients $G_i/G_{i+1}$ of an $N$-series $(*)$ have prime exponent $p$ then $L^*(G)$ can be viewed as a Lie algebra over the field with $p$ elements. 
Observe that an important example of an $N_p$-series is  the  case where the series $(*)$ is the $p$-dimension central series, also known under the name of Zassenhaus-Jennings-Lazard series (see \cite[p.\ 250]{Huppert2} for details).

Any automorphism of $G$ in the natural way induces an automorphism of $L^*(G)$. If $G$ is profinite and $\alpha$ is a coprime automorphism of $G$, then the subring (subalgebra) of fixed points of $\alpha$ in $L^*(G)$ is isomorphic with the Lie ring associated to the group $C_G(\alpha)$ via the series formed by intersections of $C_G(\alpha)$ with the series $(*)$ (see \cite{aaaa} for more details).

 In the case where the series $(*)$ is just the lower central series of $G$ we write $L(G)$ for the associated Lie ring. In the case where the series $(*)$ is the $p$-dimension central series of $G$ we write $L_p(G)$ for the subalgebra generated by the first homogeneous component $G_1/G_2$ in the associated Lie algebra over the field with $p$ elements.

We will also require the following lemma that essentially is due to Wilson and Zelmanov (cf \cite[Lemma in Section 3]{WZ}.
\begin{lemma}\label{leWZ} Let $G$ be a profinite group and $g\in G$ an element such that for any $x\in G$ there exists a positive $n$ with the property that $[x,{}_n\,g]=1$. Let $L^*(G)$ be the Lie algebra associated with $G$ using an $N_p$-series $(*)$ for some prime $p$. Then the image of $g$ in $L^*(G)$ is ad-nilpotent.
\end{lemma}

We  close this section by quoting the following result, which  is a finite version of Lazard's criterion for a pro-$p$ group to be $p$-adic analytic. The proof can be found in \cite{KS}.

\begin{theorem}\label{finiteL}
Let $P$ be a $d$-generated finite $p$-group and suppose that $L_p(G)$ is nilpotent of class $c$. Then  $P$  has a powerful characteristic subgroup of  $\{p,c,d\}$-bounded index.
\end{theorem}

Recall that powerful $p$-groups were introduced by Lubotzky and Mann in \cite{luma}.  A finite $p$-group $P$ is said to be powerful if and only if $[P,P]\leq P^{p}$ for $p \neq 2$ (or $[P,P]\leq P^{4}$ for $p=2$), where  $P^{i}$  denotes   the subgroup of $P$ generated by all $i$th powers.  Powerful  $p$-groups have some nice properties. In particular, if $P$ is a powerful $p$-group, then the subgroups $\gamma_{i}(P), P^{(i)}$ and $P^{i}$ are also powerful. Moreover, for given positive integers $n_{1},\ldots,n_{s}$, it follows, by repeated applications of \cite[Propositions 1.6 and 4.1.6]{luma}, that
$$[P^{n_{1}},\ldots,P^{n_{s}}]\leq \gamma_{s}(P)^{n_{1}\cdots n_{s}}.$$
Furthermore if a powerful $p$-group $P$ is generated by $d$ elements, then any subgroup of $P$ can be generated by at most $d$ elements and $P$ is a product of $d$ cyclic subgroups.  

      \section{Proof of Theorem A3}
We are now ready to embark on the proof of our main results.
  
\begin{proposition}\label{prop_a2} Let $G$ be an $m$-generated group satisfying the hypothesis of Theorem A3. Let $p$ be a prime number that divides the order of $G$. 
\begin{enumerate}\label{prop3}
\item There exist positive integers $e, c$ depending only on $m, n$  and $q$, such that $e\gamma_{c}(L(G))=0$;
\item $L_p(G)$ is nilpotent of $\{m,n,p,q\}$-bounded class.
\end{enumerate}
\end{proposition}

The proofs of the statements $(1)$ and $(2)$ are similar. We will give a detailed proof of the first statement. The proof of (2), with obvious changes, can be obtained simply by replacing every appeal to Theorem \ref{BazaRing} in the proof of (1) by an appeal to Corollary \ref{ZelBazaField}.  An important observation that is used in the proof is that when a finite $p$-group can be generated by $d$ elements, we can choose $d$ generators from any subset that generates the group. This follows from the well-known Burnside Basis Theorem \cite{Huppert1}. In the proof we also use the following theorem, due to Ward \cite{Ward}.
\begin{theorem}\label{ward}
Let $q$ be a prime and $A$ an elementary abelian group of order at least $q^3$. Suppose that $A$ acts coprimely on  a finite group $G$ and assume that $C_G(a)$ is nilpotent for each $a\in A^{\#}$.
Then $G$ is nilpotent.
\end{theorem}

\begin{proof}[Proof of Proposition \ref{prop3}] Recall that $G$ is an $m$-generated group acted on by $A$ in such a way that for each $a\in A^{\#}$ every element of $C_{G}(a)$ is $n$-Engel in $C_{G}(a)$. We wish to prove that there exist positive integers $e,c$ depending only on $m,n$ and $q$ such that $e\gamma_{c}(L(G))=0$. By well-known Zorn's Theorem \cite[Satz III.6.3]{Huppert1}, each $C_G(a)$ is nilpotent. Therefore Theorem \ref{ward} implies that  $G$ is nilpotent. Hence, $G$ is a direct product of its Sylow subgroups and so, without loss generality, we can assume that $G$ is a $p$-group for some prime $p\neq q$.

Let $A_1\ldots,A_s$ be the distinct maximal subgroups of $A$ and note that $s$ is a $q$-bounded number. We denote by $\gamma_j=\gamma_j(G)$ the terms of lower central series of $G$. Set $L=L(G)$ and $L_j=\gamma_j/\gamma_{j+1}$ so that $L=\oplus_{j\geq 1} L_j$. The group $A$ naturally acts on $L$. Since each subgroup $A_i$ is noncyclic, we deduce that $L=\sum_{a\in A_i^{\#}}C_L(a)$ for every $i\leq s$.

Let $L_{ij}=C_{L_j}(A_i)$. Then for any $j$ we have
$$L_j=\sum_{1\leq i\leq s}L_{ij}.$$
Further, for any $l\in L_{ij}$ there exists $x\in\gamma_j\cap C_G(A_i)$ such that $l=x\gamma_{j+1}$. Since  $x$ is an $n$-Engel element in $C_{G}(a)$, the element $l$ is ad-nilpotent of index at most $n$ in $C_L(a)$. From the fact that $L=\sum_{a\in A_i^{\#}}C_L(a)$, we deduce now that
\begin{equation}\label{first3}
\mbox{any element}\   l\  \mbox{in}\ L_{ij}\ \mbox{is ad-nilpotent in}\ L\   \mbox{with index at most} \ n.
\end{equation}

Since $G$ is generated by $m$ elements, the additive group $L_{1}$ is generated by $m$ elements. In particular, the Lie ring $L$ is generated by at most $m$ ad-nilpotent elements, each from $L_{i1}$ for some $i$.

Let $\omega$ be a primitive $q$th root of unity and consider the tensor product $\overline{L}=L\otimes\mathbb{Z}[\omega]$. Set $\overline{L}_j=L_j\otimes\mathbb{Z}[\omega]$ for $j=1,2,\ldots$. We regard $\overline{L}$ as a Lie ring and so $\overline{L}=\left<   \overline{L}_1 \right>$. We also remark that there is a natural embedding of the ring $L$ into the ring $\overline{L}.$  Since the additive subgroup $L_1$ is generated by $m$ elements, it follows that the additive subgroup $\overline{L}_1$ is generated by at most $(q-1)m$ elements. 

The group $A$ naturally acts on $\overline{L}$  and $\overline{L}_{ij}=C_{\overline{L}_j}(A_i)$, where $\overline{L}_{ij}=L_{ij}\otimes\mathbb{Z}[\omega]$. We will now prove the following claim.
\begin{equation}\label{second3}
\mbox{Any element }y\in\overline{L}_{ij}\ \mbox{is ad-nilpotent in}\ \overline{L}\ \mbox{with}\ \{n,q\}\mbox{-bounded index.}
\end{equation} 
Indeed, choose $y\in\overline{L}_{ij}$ and write 
$$y=x_0+\omega x_1+\omega^2 x_2+\cdots+\omega^{q-2}x_{q-2},$$
for suitable $x_0,x_1,\ldots,x_{q-2}\in L_{ij}$. Note that by (\ref{first3}) each of the summands $\omega^t x_t$ is ad-nilpotent in $\overline{L}$ of index at most $n$.  Denote by $K=\langle x_0, \omega x_1,\ldots, \omega^{q-2}x_{q-2} \rangle$  the subring of $\overline{L}$ generated by the elements $x_0, \omega x_1,\ldots, \omega^{q-2}x_{q-2}$. 
Since  $K$ is contained in $C_{\overline{L}}(a)$ for some $a\in A_{i}$ and $C_{G}(a)$ is $n$-Engel, we conclude that $K$ satisfies  the linearized $n$-Engel identity. Furthermore, note that a commutator of weight $k$ in the elements $\omega^{t}x_{t}$ has the form $\omega^{r}x$ for some $x$ that belongs to $L_{ih}$, where $h=kj$. By (\ref{first3}) the element $x$ is ad-nilpotent of $\{n,q\}$-bounded index and so such commutator must be ad-nilpotent of bounded index as well. Hence, by Theorem \ref{Zelring}, $K$ is nilpotent of $\{n,q\}$-bounded class. Now Lemma \ref{lemmanovo} ensures that there exists a positive integer $u$, depending only on $n$ and $q$, such that $[\overline{L},_u K]=0$. Since  $y\in K$, we conclude that $y$ is ad-nilpotent in $\overline{L}$ with $\{n,q\}$-bounded index, as claimed.

An element $x\in \overline{L}$ will be called a common ``eigenvector'' for $A$ if for any $a\in A^{\#}$ there exists a number $k$ such that $x^a=\omega^kx$. Since $(|A|,|G|)=1$, the additive group of the Lie ring $\overline{L}$ is generated by common eigenvectors for $A$ (see, for example \cite[Lemma 4.1.1]{Khu1}).

As we have already remarked the  additive  subgroup $\overline{L}_1$ is generated by at most $(q-1)m$ elements. It  follows that the Lie ring $\overline{L}$ is generated by an $\{m,q\}$-bounded number of  common eigenvectors for $A$. Certainly, the Lie ring $\overline{L}$ is generated by an $A$-invariant set of $\{m,q\}$-boundedly many  common eigenvectors for $A$.

Since $A$ is noncyclic, every common eigenvector for $A$ is contained in the centralizer $C_{\overline{L}}(A_i)$ for some $i\leq s$. Further, any commutator in common eigenvectors is again a common eigenvector. Therefore if $l_1,l_2,\ldots\in\overline{L_{1}}$ are common eigenvectors for $A$ generating $\overline{L}$, then any commutator in these generators belongs to some $\overline{L}_{ij}$ and therefore, by (\ref{second3}), is ad-nilpotent of $\{n,q\}$-bounded index.

Note that, for any $1\leq i\leq s$, the subring $C_{\overline{L}}(A_i)$ satisfies the linearized $n$-Engel identity. Thus, by Theorem \ref{BazaRing}, there exist positive integers $e$ and $c$, depending only on $m, n$ and $q$, such that $e\gamma_c(\overline{L})=0$. Since  $L$ embeds into $\overline{L}$, we also have $e\gamma_c(L)=0$ and this concludes the proof.
\end{proof}

We will require the following quantitative version of Theorem \ref{ward} that was obtained in \cite{PS-2001}.
\begin{theorem}\label{ps}
Let $q$ be a prime and $A$ an elementary abelian group of order at least $q^3$. Suppose that $A$ acts coprimely on  a finite group $G$ and assume that $C_G(a)$ is nilpotent of class at most $c$ for each $a\in A^{\#}$. Then $G$ is nilpotent  of $\{q,c\}$-bounded class.
\end{theorem}
It was recently shown in \cite{meshu} that the assumption that $A$ is abelian is actually superfluous in the above result. That is, the result holds for any group $A$ of prime exponent.
We are now ready to establish  Theorem A3.
\begin{proof}[Proof of Theorem A3] By the hypothesis, for each $a\in A^{\#}$, the centralizer $C_{G}(a)$ is $n$-Engel. Thus the Zorn theorem tells us that each $C_G(a)$ is nilpotent. We deduce from Theorem \ref{ward} that the  group $G$ is nilpotent. 

Choose  now arbitrary elements  $x,y\in G$. It is sufficient to show that the subgroup $\left<x,y\right>$ is nilpotent of $\{n,q\}$-bounded class. 
Without loss of generality we can assume that no proper $A$-invariant subgroup of $G$ contains both $x$ and $y$, and so $G=\left<x^A,y^A\right>$. Therefore $G$ can be generated by a $q$-bounded number of elements. Since $G$ is nilpotent, without loss of generality we can also assume that $G$ is a $p$-group, for some  prime $p\neq q$. 

Let $L=L(G)$. By Proposition \ref{prop3}(1) there exist positive integers $e$ and $c$ depending only on $n$ and $q$ such that $e\gamma_c(L)=0$. If $p$ is not a divisor of $e$, then $\gamma_c(L)=0$. In particular $G$ is nilpotent of class at most $c-1$, as required.  

Let us assume now that $p$ is a divisor of $e$. By Proposition \ref{prop3}(2), the Lie algebra $L_p(G)$ is nilpotent of $\{n,q\}$-bounded class. Then  Theorem \ref{finiteL} tells us that $G$ has a powerful characteristic subgroup $P$ of $\{n,q\}$-bounded index. It follows that $P$ has a bounded number of generators (see \cite[6.1.8(ii), p.\ 164]{Rob}). Therefore the rank of $P$ is $\{n,q\}$-bounded and so the rank of $G$ is $\{n,q\}$-bounded as well. This implies that any centralizer $C_G(a)$ can be generated by an $\{n,q\}$-bounded number of elements for each $a\in A^{\#}$. Zelmanov noted in \cite{Z2} that the nilpotency class of a finite $n$-Engel group is bounded in terms of $n$ and the number of generators of that group. (This means that locally nilpotent $n$-Engel groups form a variety. More results of this nature can be found in \cite{stt}). We conclude that each $C_{G}(a)$ is nilpotent with $\{n,q\}$-bounded class.  Theorem \ref{ps} implies that $G$ is nilpotent of $\{n,q\}$-bounded class. The proof is now complete. 
\end{proof}

\section{Proof of Theorem B2}
We start this section by proving Theorem B2 in the particular case where $G$ is a pro-$p$ group.

\begin{proposition}\label{le}
Let $G$ be a pro-$p$ group satisfying the hypothesis of Theorem B2. Then $G$ is locally nilpotent.
\end{proposition}
\begin{proof}

Since every finite set of $G$ is contained in a finitely generated $A$-invariant closed subgroup, we may assume that $G$ is finitely generated.  Then, of course, it will be sufficient to show that $G$ is nilpotent.

We denote by $D_j=D_j(G)$ the terms of the $p$-dimension central series of $G$. Let $L=L_p(G)$ be the Lie algebra associated with the group $G$ and 
$L_j=L\cap(D_j/D_{j+1})$. Thus, $L=\oplus_{j\geq 1}L_j$. 

The group $A$ naturally acts on $L$. Let $A_1,\ldots,A_{q+1}$ be the distinct maximal subgroups of $A$. Set $L_{ij}=C_{L_j}(A_i)$. We know that any $A$-invariant subgroup is generated by the centralizers of $A_i$. Therefore for any $j$ we have $$L_j=\sum_{i=1}^{q+1}L_{ij}.$$
Further, for any $l\in L_{ij}$ there exists an element $x\in D_j\cap C_G(A_i)$ such that $l=xD_{j+1}$. Since $x$ is Engel in $G$, it follows by Lemma \ref{leWZ} that $l$ is ad-nilpotent in $L$. Thus,

\begin{equation}\label{le1}
\mbox{any element in}\  L_{ij}  \ \mbox{is ad-nilpotent in}\  L.
\end{equation}

Let $\omega$ be a primitive $q$th root of unity and $\overline{L}=L\otimes\mathbb{F}_p[\omega]$. Here $\mathbb{F}_p$ stands for the field with $p$ elements. We can view $\overline{L}$ both as a Lie algebra over $\mathbb{F}_p$ and as that over $\mathbb{F}_p[\omega]$. It is natural to identity $L$ with the subalgebra $L\otimes 1$ of $\overline{L}$. We note that if an element $x\in L$ is ad-nilpotent of index $r$, say, then the correspondent element $x\otimes1$ is ad-nilpotent in $\overline{L}$ of the same index $r.$ 

Put $\overline{L}_j=L_j\otimes\mathbb{F}_p[\omega]$. Then $\overline{L}=\left<\overline{L}_1\right>$ and $\overline{L}$ is the direct sum of the homogeneous components $\overline{L}_j$. The group $A$ naturally acts on $\overline{L}$ and we have $\overline{L}_{ij}=C_{\overline{L}_j}(A_i)$, where $\overline{L}_{ij}=L_{ij}\otimes\mathbb{F}_p[\omega]$. Let us show that
\begin{equation}\label{le2}
\mbox{any element}\ y\in\overline{L}_{ij}\ \mbox{is ad-nilpotent in}\ \overline{L}.
\end{equation}
Since $\overline{L}_{ij}=L_{ij}\otimes\mathbb{F}_p[\omega]$, we can write
$$y=x_0+\omega x_1+\omega^2 x_2\cdots+\omega^{q-2}x_{q-2}$$
for suitable $x_0, x_1, x_2,\ldots, x_{q-2}\in L_{ij}$. In view of (\ref{le1}) it is easy to see that each of the summands $\omega^sx_s$ is ad-nilpotent in $\overline{L}$. Let $J=\left<x_0, \omega x_1, \omega^2 x_2,\ldots, \omega^{q-2}x_{q-2}\right>$ be the subalgebra of $\overline{L}$ generated by $x_0,$ $ \omega x_1,\omega^2 x_2,\ldots, \omega^{q-2}x_{q-2}$. We wish to show that $J$ is nilpotent. 

Note that $J\subseteq C_{\overline{L}}(A_i)$. A commutator of weight $t$ in the generators of $J$ has form $\omega^sx$ for some $x$ that belongs to $L_{im}$, where $m=tj$. By (\ref{le1}) the element $x$ is ad-nilpotent and so such a commutator must be ad-nilpotent. By \cite[Proposition 2.6]{PS-CA3} $L$ satisfies a multilinear polynomial identity. The multilinear identity is also satisfied in $\overline{L}$ and so it is satisfied in $J$, since $J\subseteq C_{\overline{L}}(A_i)$. Hence by Theorem \ref{Z1992} $J$ is nilpotent. By Lemma \ref{lemmanovo} (or a Lie algebra version \cite[Lemma 5]{KS}) there exists some positive integer $d$ such that $[\overline{L}, \underbrace{J,\ldots,J}_d ]=0$. This proves (\ref{le2}).

Since $A$ is abelian and the ground field is now a splitting field for $A$, every component $\overline{L}_j$ decomposes in the direct sum of common eigenspaces for $A$. In particular,  $\overline{L}_1$ is spanned by finitely many common eigenvectors for $A$, since $G$ is a finitely generated pro-$p$ group. Hence, $\overline{L}$ is generated by finitely many common eigenvectors for $A$ from $\overline{L}_1$. Since $A$ is noncyclic every common eigenvector is contained in the centralizer $C_{\overline{L}}(A_i)$ for some $i\leq q+1$.

We also note that any commutator in common eigenvectors is again a common eigenvector for $A$. Therefore, if $l_1, l_{2},\ldots\in \overline{L}_1$ are common eigenvectors for $A$ generating $\overline{L}$ then any commutator in those generators belongs to some $\overline{L}_{ij}$ and therefore, by (\ref{le2}), is ad-nilpotent.

As we have mentioned earlier, $\overline{L}$ satisfies a polynomial identity. It follows from Theorem \ref{Z1992} that $\overline{L}$ is nilpotent. Since $L$ embeds into $\overline{L}$, we  deduce that $L$ is nilpotent as well. 

According to Lazard \cite{L} the nilpotency of $L$ is equivalent to $G$ being $p$-adic analytic. By \cite[7.19 Theorem]{GA} $G$ admits a faithful linear representation over the field of $p$-adic numbers. A result of Gruenberg \cite[Theorem 0]{G} says that in a linear group the Hirsch-Plotkin radical coincides with the set of Engel elements. Since $G$ is finitely generated and $$G=\prod\limits_{a\in A^{\#}}C_G(a),$$ we deduce that $G$ is nilpotent. This concludes the proof.
\end{proof}

Now we are ready to deal with the general case of Theorem B2.
\begin{proof}[Proof of Theorem B2]
Let $\mathcal{S}$ be the subset of all $A$-invariant open normal subgroups of $G$. By \cite[Theorem 1.2.5]{W}, we have $G\cong\underleftarrow{\lim}_{N\in \mathcal{S}}\,G/N$. Note that if $N\in\mathcal{S}$ and $Q=G/N$, the group $A$ acts coprimely on $Q$ in the natural way and $Q$ is generated by the centralizers $C_Q(a)=C_G(a)N/N$. By well-known Baer's Theorem \cite[Hilfssatz III.6.14]{Huppert1}, a finite group generated by Engel elements is nilpotent and so $Q$ is nilpotent. Since this argument holds for any $N$ in $\mathcal{S}$, we conclude that $G$ is pronilpotent and therefore $G$ is the Cartesian product of its Sylow subgroups. 

Let $a\in A^{\#}$. For each positive integer $i$ define the set
$$S_i=\{(x,y)\in G\times C_G(a): [x, _iy]=1\}.$$
Since these sets are closed in $G\times C_G(a)$ and their union coincides with  $G\times C_G(a)$, by Baire category theorem \cite[p. 200]{Baire} at least one $S_{i}$ has a non-empty interior. Therefore, we can find an open normal subgroup $H$ in $G$, elements $u\in G, v\in C_G(a)$ and a positive integer $n$ such that $[ul,_nvk]=1$ for any $l\in H$ and any $k\in H\cap C_G(a)$. Let $[G:H]=m$ and let  $\pi_1=\pi(m)$ be the set of primes dividing $m$. Denote  $O_{\pi'_1}(G)$ by $T$ and observe that for all $x\in C_T(a)$ we have $[T,_nx]=1$.

Of course, the open subgroup $H$, the set $\pi_1$ and the integer $n$ depend  on the choice of $a\in A^{\#}$ so strictly speaking they should be denoted by $H_a$, $\pi_a$ and $n_a$, respectively. We choose such $H_a$, $\pi_a$ and $n_a$ for every $a\in A^{\#}$. Set $\pi=\cup_{a\in A^{\#}}\pi_a$, $n=max\{n_a:a\in A^{\#}\}$ and $R=O_{\pi'}(G).$ Note that $\pi$ is a finite set. The choice of the set $\pi$ guarantees that for each $a\in A^{\#}$ every element of $C_R(a)$ is $n$-Engel in $R$. Using the routine inverse limit argument we deduce from Theorem A2 that $R$ is $k$-Engel for some suitable integer $k$. Thus, by \cite[Theorem 5]{WZ}, the subgroup $R$ is locally nilpotent. Let $p_1,\ldots, p_r$ be the finitely many primes in $\pi$ and $P_1,\ldots, P_r$ be the corresponding Sylow subgroups of $G$. Then $G=P_1\times\ldots\times P_r\times R$ and therefore it is sufficient to show that each subgroup $P_i$ is locally nilpotent. This is immediate from Proposition \ref{le}. The proof is complete.
\end{proof}


\end{document}